\documentclass[11pt]{article}
\usepackage{enumerate}
\usepackage{amssymb,a4wide,latexsym,makeidx,epsfig,fleqn}
\usepackage{amsthm}
\usepackage{amsmath}
\usepackage{enumerate}
\usepackage{extarrows}
\usepackage{graphicx}
\usepackage{subfigure}
\usepackage{float}
\usepackage[justification=centering]{caption}
\newtheorem{theorem}{Theorem}[section]

\newtheorem{definition}[theorem]{Definition}
\newtheorem{lemma}[theorem]{Lemma}

\newtheorem{corollary}[theorem]{Corollary}

\begin{document}
\textwidth 150mm \textheight 225mm
\title{Matching, odd $[1,b]$-factor and distance spectral radius of graphs with given some parameters
\thanks{Supported by the National Natural Science Foundation of China (Nos. 12001434 and 12271439), the Natural Science Basic Research Program of Shaanxi Province (Nos. 2023-JC-YB-070 and 2024JC-YBQN-0015).}}
\author{Zengzhao Xu$^{a}$, Weige Xi$^{a}$\footnote{Corresponding author.}, Ligong Wang$^{b}$\\
{\small $^{a}$ College of Science, Northwest A\&F University, Yangling, Shaanxi 712100, China}\\
{\small $^{b}$ School of Mathematics and Statistics, Northwestern Polytechnical University,}\\
{\small  Xi'an, Shaanxi 710129, P.R. China.}\\
{\small E-mail: xuzz0130@163.com; xiyanxwg@163.com; lgwangmath@163.com}\\}
\date{}
\maketitle
\begin{center}
\begin{minipage}{120mm}
\vskip 0.3cm
\begin{center}
{\small {\bf Abstract}}
\end{center}
{\small  For a connected graph $G$, let $\mu(G)$ denote the distance spectral radius of $G$. A matching in a graph $G$ is a set of disjoint edges of $G$. The maximum size of a matching in $G$ is called the matching number of $G$, denoted by $\alpha(G)$. An odd $[1, b]$-factor of a graph $G$ is a spanning subgraph $G_0$ such that the degree $d_{G_0}(v)$ of $v$ in $G_0$ is odd and $1\le d_{G_0}(v)\le b$ for every vertex $v\in V (G)$. In this paper, we give a sharp upper bound in terms of the distance spectral radius to guarantee $\alpha(G)>\frac{n-k}{2}$ in an $n$-vertex $t$-connected graph $G$, where $2\le k \le n-2$ is an integer. We also present a sharp upper bound in terms of distance spectral radius for the existence of an odd $[1,b]$-factor in a graph with given minimum degree $\delta$.

}
\end{minipage}
\end{center}

\section{Introduction }

 All graphs considered are finite, simple, and connected throughout this paper. For a graph $G$, we use $V(G)=\{v_1,v_2,\ldots, v_n\}$ and $E(G)$ to denote the vertex set and the edge set of $G$, respectively. The set of neighbors of the vertex $v_i$ is denoted by $N_G(v_i)$, which is defined as the set of vertices adjacent to $v_i$. The degree of the vertex $v_i$ in $G$ is the number of its neighbours, denoted by $d_G(v_i)$ (or simply $d(v_i)$), i.e., $d_G(v_i)=|N_G(v_i)|$. The minimum degree of $G$ is denoted by $\delta(G)$ (or simply $\delta$). For any two graphs $G_1$ and $G_2$, we use $G_1+ G_2$ to denote the disjoint union of $G_1$ and $G_2$. The join of $G_1$ and $G_2$, denoted by $G_1\vee G_2$, is the graph obtained from $G_1+G_2$ by adding all possible edges between $V(G_1)$ and $V(G_2)$. For $S\subseteq V(G)$, we use $G-S$ to denote the subgraph obtained from $G$ by deleting the vertices in $S$ together with their incident edges. For $E'\subseteq E(G)$, we use $G-E'$ to denote the subgraph obtained from $G$ by deleting the edges in $E'$. Let $K_n$ denote an $n$-vertex complete graph. The connectivity of $G$ is the minimum number of vertices whose deletion induces a non-connected graph or a single vertex. For $t\ge 0$, a graph $G$ is called $t$-connected if the connectivity of $G$ is at least $t$.

For a connected graph $G$ of order $n$, the distance between vertices $v_i$ and $v_j$ denoted by $d_{ij}$, is the length of the shortest path between $v_i$ and $v_j$. The distance matrix of $G$ is defined as $D(G)=(d_{ij})_{n\times n}$, where $(i, j)$-entry is $d_{ij}$. The distance spectral radius of $G$ is the largest eigenvalues of $D(G)$, denoted by $\mu(G)$.

For a graph $G$, a matching of $G$ is a set of pairwise nonadjacent edges of $G$. The maximum size of a matching in $G$ is called matching number of $G$, denoted by $\alpha(G)$. A vertex $v$ is said to be $M$-saturated if $v$ is incident to some edge of $M$. A matching $M$ is called perfect matching if every vertex of $G$ is $M$-saturated. Therefore, if a graph $G$ contains a perfect matching, it must have an even number of vertices and $\alpha(G)=\frac{|V(G)|}{2}$.

One of our results is to characterize the matching number of a graph using the distance spectral radius. Studying the matching of graphs using the spectral radius has received a lot of attention of researchers in recent years. For example, Feng et al. \cite{FYZ} gave a spectral radius condition for a graph with given matching number. O \cite{O} proved a lower bound for the spectral radius in an $n$-vertex graph to guarantee the existence of a perfect matching. Zhang \cite{Z} characterized the extremal graphs with maximum spectral radius among all $t$-connected graphs on $n$ vertices with matching number at most $\frac{n-k}{2}$, where $2\le k \le n-2$ is an integer. Liu et al. \cite{LYL} extended some results of \cite{O} and \cite{Z}, they proved sharp upper bounds for spectral radius of $A_{\alpha}(G)$ in an $n$-vertex $t$-connected graph with the matching number at most $\frac{n-k}{2}$. Zhang and van Dam \cite{ZD} gave a sufficient condition in terms of distance spectral radius for the $k$-extendability of a graph and completely characterized the corresponding extremal graphs. Guo et al. \cite{GLMM} gave a spectral condition for a graph to have a rainbow matching. For more literature on studying the matching of graphs using the spectral radius, please refer to \cite{HLZ,Il,KOSS,SYZL,ZL2}.

Inspired by \cite{LYL} and \cite{Z}, in this paper we firstly investigate the relation between the distance spectral radius of an $n$-vertex $t$-connected graph and its matching number.

$[a, b]$-factor plays important roles in solving the graph decomposability problem. An $[a, b]$-factor of a graph $G$ is defined as a spanning subgraph $G_0$ such that $a\le d_{G_0}(v)\le b$ for each $v\in V (G)$. An odd $[1,b]$-factor of a graph $G$ is defined as a spanning subgraph $G_0$ such that $d_{G_0}(v)$ is odd and $1\le d_{G_0}(v)\le b$ for each $v\in V (G)$. Obviously, a perfect matching is a special odd $[1,b]$-factor when $b=1$.

Recently, the existence of an $[a, b]$-factor in a graph has been investigated by many researchers. In \cite{O2}, O provided some conditions for the existence of an $[a, b]$-factor in an $h$-edge-connected $r$-regular graph. In \cite{FLL}, Fan et al. provided spectral conditions for the existence of an odd $[1, b]$-factor in a connected graph with minimum degree $\delta$ and the existence of an $[a, b]$-factor in
a graph, respectively. In \cite{LM}, Li and Miao considered the edge condition for a connected graph to contain an odd $[1, b]$-factor. For more literature on $[a, b]$-factor, please refer to \cite{HLL,ZL4,ZL3}

Motivated by \cite{FLL}, in this paper we provide a condition in terms of distance spectral radius for the existence of an odd $[1,b]$-factor in a graph with given minimum degree.

The rest of the paper is structured as follows. In Section 2, we recall some important known concepts and lemmas to prove the theorems in the following sections. In Section 3, we give a sharp upper bound in terms of the distance spectral radius to guarantee $\alpha(G)>\frac{n-k}{2}$ in an $n$-vertex $t$-connected graph $G$, where $2\le k \le n-2$ is an integer. In Section 4, we provide a sharp upper bound in terms of distance spectral radius for the existence of an odd $[1,b]$-factor in a graph with given minimum degree $\delta$.

\section{Preliminaries}

\quad\quad In this section, we give some concepts and useful lemmas which will be used in the follows. First of all, we give some known lemmas about matching number and $[1,b]$-factor in a graph $G$. Moreover, for any $S\subseteq V(G)$ of a graph $G$, $o(G-S)$ denotes the number of odd components in graph $G-S$.

\begin{lemma}(\cite{LP})\label{le:1}  \ Let $G$ be a graph of order $n$. Then
	$$\alpha(G)=\frac{1}{2}(n-\max\{o(G-S)-|S|: \textrm{for all} \ S \subseteq V(G)\}).$$
\end{lemma}

\begin{lemma}(\cite{A})\label{le:2} \ Let $G$ be a graph and $b$ be a positive odd integer. Then $G$ contains an odd $[1, b]$-factor if and only if for every $S\subseteq V(G)$,
$$o(G-S)\le b|S|.$$
\end{lemma}

Equitable quotient matrix plays an important role in the study of spectral graph theory. Thus, we will give the definition of the equitable quotient matrix and its some useful properties.

\noindent\begin{definition}(\cite{YYSX})\label{D1} Let $M$ be a complex matrix of order $n$ described in the following block form
	\begin{equation*}
		M=\begin{bmatrix}
			M_{11} & \cdots & M_{1t} \\
			\vdots& \ddots & \vdots \\
			M_{t1} &\cdots & M_{tt}
		\end{bmatrix},
	\end{equation*}
where the blocks $M_{ij}$ are the $n_i\times n_j$ matrices for any $1\le i,j\le t$ and $n=n_1+n_2+\cdots+n_t.$ For $1\le i,j\le t$, let $b_{ij}$ denote the average row sum of $M_{ij}$, i.e. $b_{ij}$ is the sum of all entries in $M_{ij}$ divided by the number of rows. Then $B(M)=(b_{ij})$(or simply $B$) is called the quotient matrix of $M$. If, in addition, for each pair $i$, $j$, $M_{ij}$ has a constant row sum, then $B$ is called the equitable quotient matrix of $M$.
\end{definition}

\noindent\begin{lemma}(\cite{YYSX})\label{le:10} Let $B$ be the equitable quotient matrix of $M$, where $M$ is as shown in Definition \ref{D1}. In addition, let $M$ be a nonnegative matrix. Then the spectral radius relation satisfies $\rho(B)=\rho(M)$, where $\rho(B)$ and $\rho(M)$ denote the spectral radii of $B$ and $M$ respectively.
\end{lemma}

Finally, we give some known results of the change of distance spectral radius caused by graph transformation.

\begin{lemma}(\cite{G})\label{le:8} \ Let $e$ be an edge of $G$ such that $G-e$ is connected. Then
	$\mu(G)<\mu(G-e)$.
\end{lemma}

\begin{lemma}(\cite{ZL})\label{le:3} \
	Let $n$, $c$, $s$ and $n_i (1\le i\le c)$ be positive integers with $n_1\ge n_2\ge \cdots \ge n_c\ge 1$ and $n_1+n_2+\cdots +n_c=n-s$. Then
	$$\mu(K_s \vee (K_{n_1}+K_{n_2}+\cdots +K_{n_c}))\ge\mu(K_s \vee (K_{n-s-(c-1)}+(c-1)K_1)),$$
with equality if and only if $(n_1,n_2,\cdots,n_c)=(n-s-(c-1),1,\cdots,1)$.
\end{lemma}

\begin{lemma}(\cite{LLS})\label{le:4} \
	Let $n$, $c$, $s$, $p$ and $n_i (1\le i\le c)$ be positive integers with $n_1\ge2p$, $n_1\ge n_2\ge \cdots \ge n_c\ge p$ and $n_1+n_2+\cdots +n_c=n-s.$ Then
	$$\mu(K_s \vee (K_{n_1}+K_{n_2}+\cdots + K_{n_c}))\ge\mu(K_s \vee (K_{n-s-p(c-1)}+(c-1)K_p)),$$ with equality if and only if $(n_1,n_2,\cdots,n_c)=(n-s-p(c-1),p,\cdots,p)$.
\end{lemma}

\section{Matching number and distance spectral radius of $t$-connected graphs}

\quad\quad First of all, we give a lemma which was proposed by Zhang \cite{Z}.

\begin{lemma}(\cite{Z})\label{le:3.1}
	 Let $G$ be a connected graph on $n$ vertices with connectivity $t(G)$ and matching number $\alpha(G)<\lfloor \frac{n}{2} \rfloor$ (implying that $n-2\alpha(G)\ge2$). Then $t(G)\le \alpha(G)$.
\end{lemma}

Let $G$ be a $t$-connected graph on $n$ vertices with $\alpha(G)\le\frac{n-k}{2}$, where $2\le k \le n-2$ is an integer, and let $S\subseteq V(G)$ be a vertex subset such that
$\alpha(G)=\frac{1}{2}(n-(o(G-S)-|S|))$. Based on Lemma \ref{le:3.1}, we have $t\le |S|\le \alpha(G)\le\frac{n-k}{2}$. It is natural to consider the following question.

{\bf Question 3.2:} Can we find a condition in terms of distance spectral radius that makes the matching number of an $n$-vertex $t$-connected graph $G$ more than $\frac{n-k}{2}$, where $2\le k \le n-2$ is an integer? In addition, can we characterize the corresponding spectral extremal graphs?

Based on the question, we give the following theorem.

\noindent\begin{theorem}\label{T1.5}  \  Let $n$, $t$ and $k$ be three positive integers, where $2\le k \le n-2$, $1\le t\le \frac{n-k}{2}$ and $n\equiv k(mod\ 2)$. Let $G$ be a $t$-connected graph of order $n\ge9k+10t-11$, and let $\alpha(G)$ be the matching number of $G$. If $\mu(G)\le \mu(K_{t} \vee (K_{n+1-2t-k}+(t+k-1)K_1))$, then $\alpha(G)>\frac{n-k}{2}$ unless $G\cong K_{t} \vee (K_{n+1-2t-k}+(t+k-1)K_1)$. (see Figure 1)
\end{theorem}

\begin{figure}[H]
	\begin{centering}
		\includegraphics[scale=0.8]{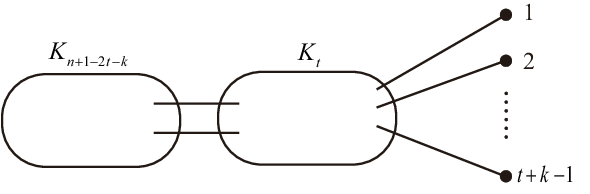}
		\caption{The extremal graph of Theorem 3.2.}\label{Fig.1.}
	\end{centering}
\end{figure}

Before we prove Theorem \ref{T1.5}, we will prove the following lemmas.

\noindent\begin{lemma}\label{le:9}  \  Let $n$, $t$ and $k$ be three positive integers, where $2\le k \le n-2$, $1\le t\le \frac{n-k}{2}$ and $n\equiv k(mod\ 2)$. Let $G$ be a $t$-connected graph of order $n$ with matching number $\alpha(G)$. If $\alpha(G)\le \frac{n-k}{2}$, then $\mu(G)\ge \mu(K_{s} \vee (K_{n+1-2s-k}+(s+k-1)K_1))$, where $t\le s\le\frac{n-k}{2}$. Equality holds if and only if $G\cong K_{s} \vee (K_{n+1-2s-k}+(s+k-1)K_1).$
\end{lemma}

\begin{proof} Suppose that the distance spectral radius of $G$ is as small as possible among all $t$-connected graph on $n$ vertices with matching number $\alpha\le \frac{n-k}{2}$. By Lemma \ref{le:1}, there
exists a vertex subset $S\subseteq V(G)$ such that $\alpha(G)=\frac{1}{2}(n-(o(G-S)-|S|))$. Then, by Lemma \ref{le:8}, we can claim that all components of $G-S$ are odd components. Otherwise, we can randomly remove one vertex from each even component of $G-S$ to the set $S$ until all components of $G-S$ are odd components. In this process, it can be checked that the number of vertices in set $S$ and the number of odd components $G-S$ have the same increase. Therefore, the equality $o(G-S)-|S|=n-2\alpha(G)$ always holds.
	
Let $s=|S|$ and $q=o(G-S)$. Since $o(G-S)-|S|=n-2\alpha(G)\ge k$, we have $q\ge s+k$.  Then we will prove the following claims.

{\bf Claim 1.} Let $G_1=K_s \vee (K_{n_1}+K_{n_2}+\cdots+K_{n_q})$, where $n_1\ge n_2\ge \cdots \ge n_q$ are positive odd integers. Then $\alpha(G_1)\le \frac{n-k}{2}$ and $\mu(G)\ge \mu(G_1)$ with the equality holds if and only if $G\cong G_1$.

{\bf Proof.} Obviously, $G$ is a spanning subgraph of $G_1$. By Lemma \ref{le:8}, $\mu(G)\ge \mu(G_1)$,
where equality holds if and only if $G\cong G_1$. Note that $o(G_1-S)=o(G-S)\ge s + k$ and $n-2\alpha(G_1)=\max\{o(G_1-K)-|K|: \textrm{for all} \ K \subseteq V(G_1)\}\ge o(G_1-S)-|S|\ge k$, we get $\alpha(G_1)\le \frac{n-k}{2}$.

{\bf Claim 2.} Let $G_2=K_s \vee (K_{n'_1}+K_{n'_2}+\cdots+K_{n'_{s+k}})$, where $n'_1=n_1+\sum_{i=s+k+1}^{q}n_i$ and $n'_i=n_i$ for $i=2,\cdots,s+k$. Then $\alpha(G_2)\le \frac{n-k}{2}$ and $\mu(G_1)\ge \mu(G_2)$ with the equality holds if and only if $G_1\cong G_2$.

{\bf Proof.} For $i=1,2,\cdots,q$, Since $n_i$ is odd, we can take $n_i=2k_i+1$, where $k_i\ge0$ and $k_i$ is integer. Since $s+\sum_{i=1}^{q}n_i=s+q+\sum_{i=1}^{q}(2k_i)=n$, we have $q+s\equiv n\equiv k(mod\ 2)$. Thus $q-s-k=q+s-k-2s$ is even and $n'_1=n_1+\sum_{i=s+k+1}^{q}n_i$ is odd. Obviously, $o(G_2-S)=o(G_1-S)-(q-s-k)=s+k$ and $n-2\alpha(G_2)\ge o(G_2-S)-|S|$. Hence $\alpha(G_2)\le \frac{n-k}{2}$. Since $G_1$ is a spanning subgraph of $G_2$, by Lemma \ref{le:8}, $\mu(G_1)\ge \mu(G_2)$, where the equality holds if and only if $G_1\cong G_2$.	

{\bf Claim 3.} Let $G_3=K_s \vee (K_{n+1-2s-k}+(s+k-1)K_1)$. Then $\alpha(G_3)\le \frac{n-k}{2}$ and $\mu(G_2)\ge \mu(G_3)$ with the equality holds if and only if $G_2\cong G_3$.

{\bf Proof.}  Obviously, $o(G_3-S)=o(G_2-S)=s + k$ and $n-2\alpha(G_3)\ge o(G_3-S)-|S|$. Therefore $\alpha(G_3)\le \frac{n-k}{2}$. Moreover, by Lemma \ref{le:3}, $\mu(G_2)\ge \mu(G_3)$, where the equality holds if and only if $G_2\cong G_3$.	

Based on the above results, we can conclude that if $G$ is a $t$-connected graph of order $n$ with $\alpha(G)\le \frac{n-k}{2}$, then $\mu(G)\ge \mu(G_3)=\mu(K_s \vee (K_{n+1-2s-k}+(s+k-1)K_1))$ with the equality holds if and only if $G\cong K_s \vee (K_{n+1-2s-k}+(s+k-1)K_1)$. This completes the proof.
\end{proof}

\noindent\begin{lemma}\label{le:20}  \  Let $n\ge9k+10t-11$, $t$ and $k$ be three positive integers, where $2\le k \le n-2$, $1\le t\le \frac{n-k}{2}$ and $n\equiv k(mod\ 2)$. Then $\mu(K_s \vee (K_{n+1-2s-k}+(s+k-1)K_1))\ge \mu(K_{t} \vee (K_{n+1-2t-k}+(t+k-1)K_1))$, where $t\le s\le\frac{n-k}{2}$. Equality holds if and only if $K_s \vee (K_{n+1-2s-k}+(s+k-1)K_1)\cong K_{t} \vee (K_{n+1-2t-k}+(t+k-1)K_1).$
\end{lemma}

\begin{proof} For convenience, let $G_s=K_s \vee (K_{n+1-2s-k}+(s+k-1)K_1)$ and $G_t=K_{t} \vee (K_{n+1-2t-k}+(t+k-1)K_1)$. Since $t \le s\le \frac{n-k}{2}$, then we will discuss the proof in two ways according to the value of $s$.
	
{\bf Case 1.} $s=t$.
	
Then $G_s\cong G_t$. Clearly, the result holds.
	
{\bf Case 2.} $t+1\le s\le \frac{n-k}{2}.$
	
We divide $V(G_s)$ into three parts: $V(K_s)$, $V(K_{n+1-2s-k})$ and $V((s+k-1)K_1)$. Then the distance matrix of $G_s$, denoted by $D(G_s)$, is
	\begin{equation*}
		\begin{bmatrix}
			(J-I)_{s\times s} & J_{s\times (n+1-2s-k)} & J_{s\times (s+k-1)} \\
			J_{(n+1-2s-k)\times s}& (J-I)_{(n+1-2s-k)\times (n+1-2s-k)} & 2J_{(n+1-2s-k)\times (s+k-1)}\\
			J_{(s+k-1)\times s}& {2J}_{(s+k-1)\times (n+1-2s-k)} & 2(J-I)_{(s+k-1)\times (s+k-1)}
		\end{bmatrix},
	\end{equation*}
where $J_{i\times j}$ denotes the ${i\times j}$ all-one matrix and $I_{i\times i}$ denotes the ${i\times i}$ identity square matrix. Then the equitable quotient matrix of the distance matrix $D(G_s)$, denoted by $M_s$, with respect to the partition $V(K_s)\cup V(K_{n+1-2s-k})\cup V((s+k-1)K_1)$ is
	\begin{equation*}
		M_s=\begin{bmatrix}
			s-1 & n+1-2s-k & s+k-1 \\
			s & n-2s-k & 2(s+k-1) \\
			s & 2(n+1-2s-k) & 2(s+k-2)
		\end{bmatrix}.
	\end{equation*}
Through a simple calculation, the characteristic polynomial of $M_s$ is
\begin{align*}
f_s(x)&=x^3+(-s-n-k+5)x^2+(5s^2+(-2n+7k-8)s-2kn-n+2k^2-5k+8)x\\
	& \ \ \ -2s^3+(n-3k+8)s^2+(kn-3n-k^2+9k-8)s-2kn+2k^2-4k+4.
\end{align*}	
We use $y_1(M_s)$ to denote the largest root of the equation $f_s(x)=0$. By Lemma \ref{le:10}, $\mu(G_s)=y_1(M_s)$. What's more, we can get the equitable quotient matrix $M_{t}$ of $G_t=K_{t} \vee (K_{n+1-2t-k}+(t+k-1)K_1)$ by replacing $s$ with $t$. Similarly, we can get the characteristic polynomial $f_{t}(x)$ of $M_{t}$ and $\mu(G_t)=y_1(M_{t})$ is the largest root of the equation $f_{t}(x)=0$. By direct calculation, we have
	\begin{align*}
f_s(x)-f_{t}(x)&=(t-s)[x^2+(2n+8-5(t+s)-7k)x+2s^2+(2t-n+3k-8)s\\
&\ \ \ \ +2t^2+(-n+3k-8)t+3n-kn+k^2-9k+8].
\end{align*}	
Obviously, $G_s$ and $G_t$ are both spanning subgraphs of $K_n$, by Lemma \ref{le:8}, $\mu(G_s)>\mu(K_{n})=n-1$ and $\mu(G_t)>\mu(K_{n})=n-1$. Then we will give the proof that $f_s(x)-f_{t}(x)<0$ for $x\in[n-1,+\infty)$. Sine, $t<s$, thus we only need to prove that $p(x)>0$ for $x\in[n-1,+\infty)$, where
	$$p(x)=x^2+(2n+8-5(t+s)-7k)x+2s^2(2t-n+3k-8)s+2t^2+(-n+3k-8)t+3n-kn+k^2-9k+8.$$
Since the symmetry axis of $p(x)$ is
	\begin{align*}
		\hat{x}&=\frac{5(t+s)+7k-2n-8}{2}\\
		&=\frac{5}{2}s+\frac{5}{2}t+\frac{7}{2}k-n-4\\
		&\le\frac{5}{4}(n-k)+\frac{5}{2}t+\frac{7}{2}k-n-4\\
		&=\frac{1}{4}n+\frac{5}{2}t+\frac{9}{4}k-4,
	\end{align*}
note that $n\ge9k+10t-11>3k+\frac{10}{3}t-4$ and $n>3k+\frac{10}{3}t-4\iff \frac{1}{4}n+\frac{5}{2}t+\frac{9}{4}k-4<n-1$, we get $\frac{5(t+s)+7k-2n-8}{2}<n-1$. Thus, $p(x)$ is increasing with respect to $x\in[n-1,+\infty)$, and
	$$p(x)\ge p(n-1)=2s^2+(3k+2t-6n-3)s+3n^2+2t^2+(3k-6n-3)t+7n-8kn+k^2-2k+1.$$
	Let$$v(s)\triangleq p(n-1)=2s^2+(3k+2t-6n-3)s+3n^2+2t^2+(3k-6n-3)t+7n-8kn+k^2-2k+1.$$
Recall that $t+1 \le s\le \frac{n-k}{2}$ and $n\ge9k+10t-11$, then
	\begin{align*}
		\frac{dv}{ds}&=4s+3k+2t-6n-3\\
		&\le2n-2k+3k+2t-6n-3\\
		&=-4n+k+2t-3<0.
	\end{align*}
Thus, $v(s)$ is decreasing with respect to $s\in[t+1,\frac{n-k}{2}]$. Furthermore,
\begin{align*}
	v(s)\ge v(\frac{n-k}{2})&=\frac{1}{2}[n^2-(10t+9k-11)n+4t^2+(4k-6)t-k+2]\\
	&>\frac{1}{2}[n^2-(10t+9k-11)n].
\end{align*}
Note that $n\ge10t+9k-11$, we get $v(s)>0$. Therefore, $p(x)\ge p(n-1)=v(s)>0$, which implies $f_s(x)<f_{t}(x)$ for $x\in[n-1,+\infty)$. In addition, by $\min\{\mu(G_s),\mu(G_t)\}>n-1$, we get that $\mu(G_s)>\mu(G_t)$. This completes the proof.
\end{proof}

By Lemma \ref{le:9} and Lemma \ref{le:20}, Theorem \ref{T1.5} clearly holds.

Let $k=2$ in Theorem \ref{T1.5}, we can get a condition in terms of distance spectral radius in a graph $G$ with order $n$ and connectivity $t$ such that $\alpha(G)>\frac{n}{2}-1$. In addition, the matching number of $G$ is no more then $\frac{n}{2}$. Hence, we have $\alpha(G)=\frac{n}{2}$. Then we can obtain the following corollary about the perfect matching based on the distance spectral radius.

\noindent\begin{corollary}\label{C1.3}  \ Let $n$ be an even integer. Suppose that $t$ is a positive integer, where $1\le t\le \frac{n-2}{2}$. Let $G$ be a graph of order $n\ge10t+7$ with connectivity $t$. If $\mu(G)\le \mu(K_{t} \vee (K_{n-2t-1}+(t+1)K_1))$, then $G$ contains a perfect matching unless $G\cong K_{t} \vee (K_{n-2t-1}+(t+1)K_1)$.
\end{corollary}

\section{Odd $[1,b]$-factor and distance spectral radius of graph with given minimum degree}

\quad\quad Amahashi \cite{A} gave a sufficient and necessary condition for a graph contains an odd $[1,b]$-factor. Therefore, it is natural to consider the following question.

{\bf Question 4.1:} Can we obtain a distance spectral radius condition that makes a graph $G$ with given minimum degree $\delta$ having an odd $[1,b]$-factor? In addition, can we characterize the corresponding spectral extremal graphs?

Based on the question, we give the following theorem.

\noindent\begin{theorem}\label{T1.4}  \ Let $G$ be a connected graph of even order
 $n\ge\max\{ 2b\delta^2,(\frac{3}{b}+5+2b)\delta+1+\frac{3(b+1)}{b^2}\}$ with minimum degree $\delta\ge3$, where $b$ is a positive odd integer. If $\mu(G)\le \mu(K_{\delta} \vee (K_{n-(b+1)\delta-1}+(b\delta+1)K_1))$, then $G$ has an odd $[1,b]$-factor unless $G\cong K_{\delta} \vee (K_{n-(b+1)\delta-1}+(b\delta+1)K_1)$.(see Figure 2)
\end{theorem}

\begin{figure}[H]
	\begin{centering}
		\includegraphics[scale=0.8]{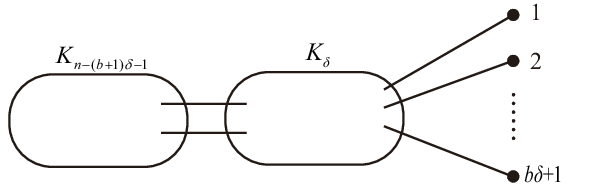}
		\caption{The extremal graph of Theorem 4.1.}\label{Fig.2.}
	\end{centering}
\end{figure}

\begin{proof} Let $G$ be a connected graph of even order
$n\ge\max\{ 2b\delta^2,(\frac{3}{b}+5+2b)\delta+1+\frac{3(b+1)}{b^2}\}$, where $\delta$ is the minimum degree of $G$ and $\delta\ge3$, $b$ is a positive odd integer. Suppose to the contrary that $G$ has no odd $[1,b]$-factor. Then by Lemma \ref{le:2}, there exists a vertex subset $S\subseteq V(G)$ such that $o(G-S)> b|S|$.
Let $|S|=s$ and $o(G-S)=q$. Since $n$ is even, it is easy to see that $s$ and $q$ have the same parity. Since $b$ is odd, we have that $bs$ and $q$ have the same parity. Thus $q=o(G-S)\ge bs+2$. Let $t=bs+2$. Since $s+t\le s+q\le n$, we have $s\le\frac{n-2}{b+1}$. Obviously, $G$ is a spanning subgraph of $K_s \vee (K_{n_1}+K_{n_2}+\cdots+ K_{n_t})$ for some positive odd integer $n_1\ge n_2\ge \cdots\ge n_t$ with $\sum_{i=1}^{t}n_i+s=n$. Let $G_1=K_s \vee (K_{n_1}+K_{n_2}+\cdots+ K_{n_t})$. By Lemma \ref{le:8},
\begin{equation}\label{eq2}
	\mu(G)\ge \mu(G_1),
\end{equation}
where equality holds if and only if $G\cong G_1$. Let $G_{\delta}=K_{\delta} \vee (K_{n-(b+1)\delta-1}+(b\delta+1)K_1)$. Let $S_1$ be the vertex set of $K_s$. Since there exists a vertex subset $S_1\subseteq V(G_1)$ such that $o(G_1-S_1)=t=bs+2>bs=b|S_1|$, we have that $G_1$ contains no odd $[1,b]$-factor. Recall that $s\le \frac{n-2}{b+1}$, then we will discuss the proof in three ways according to the value of $s$.

{\bf Case 1.} $s=\delta.$

In this case, we have $G_1=K_s \vee (K_{n_1}+K_{n_2}+\cdots+ K_{n_t})=K_{\delta} \vee (K_{n_1}+K_{n_2}+\cdots+ K_{n_{b\delta+2}})$.

By Lemma \ref{le:3},
\begin{equation}\label{e:1}
	\mu(G_{\delta})=\mu(K_{\delta} \vee (K_{n-(b+1)\delta-1}+(b\delta+1)K_1))\le \mu(G_1),
\end{equation}
with equality if and only if $G_{\delta}\cong G_1$.

Furthermore, combining with (\ref{eq2}) and (\ref{e:1}),
$$\mu(G)\ge\mu(G_{\delta})=\mu(K_{\delta} \vee (K_{n-(b+1)\delta-1}+(b\delta+1)K_1)),$$
with equality if and only if $G_{\delta}\cong G$.

Moreover, according to the assumed condition $\mu(G)\le \mu(K_{\delta} \vee (K_{n-(b+1)\delta-1}+(b\delta+1)K_1))$, we have $\mu(G)= \mu(K_{\delta} \vee (K_{n-(b+1)\delta-1}+(b\delta+1)K_1))$.

Based on the above results, we conclude that $G\cong K_{\delta} \vee (K_{n-(b+1)\delta-1}+(b\delta+1)K_1))$. In addition, take $S=V(K_{\delta})$, then $o(G-S)=b\delta+2=bs+2>bs=b|S|$, which implies $G$ does not have an odd $[1,b]$-factor.

{\bf Case 2.} $\delta+1\le s\le \frac{n-2}{b+1}.$

Let $G_s=K_{s} \vee (K_{n-(b+1)s-1}+(bs+1)K_1)$. By Lemma \ref{le:3},
\begin{equation}\label{e:2}
	\mu(G_{s})=\mu(K_{s} \vee (K_{n-(b+1)s-1}+(bs+1)K_1))\le \mu(G_1),
\end{equation}
with equality if and only if $G_s\cong G_1$. For the graph $G_s$, let $D(G_s)$ denote the distance matrix of $G_s$. Then the equitable quotient matrix of $D(G_s)$, denoted by $M_s$, with respect to the partition $V(K_s)\cup V(K_{n-(b+1)s-1})\cup V((bs+1)K_1)$ is
\begin{equation*}
	M_s=\begin{bmatrix}
		s-1 & n-(b+1)s-1 & bs+1 \\
		s & n-(b+1)s-2 & 2(bs+1) \\
		s & 2(n-(b+1)s-1) & 2bs
	\end{bmatrix},
\end{equation*}
and the characteristic polynomial of $M_s$ is
\begin{align*}
	f_s(x)&=x^3+(-bs-n+3)x^2+(2b^2s^2+3bs^2-2bns+3bs+3s-5n+6)x-(b^2+b)s^3\\
	&\ \ \ \ +(bn+2b^2+b-1)s^2+(n-2bn+4b+2)s+4-4n.
\end{align*}
We use $y_1(M_s)$ to denote the largest root of the equation $f_s(x)=0$. By Lemma \ref{le:10}, we can get $\mu(G_s)=y_1(M_s)$. Moreover, we obtain the equitable quotient matrix $M_{\delta}$ of $G_{\delta}=K_{\delta} \vee (K_{n-(b+1)\delta-1}+(b\delta+1)K_1)$ by replacing $s$ with $\delta$. Similarly, we can get the characteristic polynomial $f_{\delta}(x)$ of $M_{\delta}$ and $\mu(G_\delta)=y_1(M_{\delta})$ is the largest root of the equation $f_{\delta}(x)=0$. By a calculation, we have
\begin{align*}
	f_s(x)-f_{\delta}(x)&=(\delta-s)[bx^2+(-2b^2s-3bs+2bn-2b^2\delta-3b\delta-3b-3)x\\
	&\ \ \ +(b^2s+bs-bn+b^2\delta+b\delta-2b^2-b+1)s\\
	&\ \ \ +(-bn+b^2\delta+b\delta-2b^2-b+1)\delta-n+2bn-4b-2].
\end{align*}	
Since $G_s$ and $G_{\delta}$ are spanning subgraphs of $K_n$, by lemma \ref{le:8}, $\mu(G_s)>\mu(K_{n})=n-1$ and $\mu(G_{\delta})>\mu(K_{n})=n-1$.
Then we will prove that $f_s(x)-f_{\delta}(x)<0$ for $x\in[n-1,+\infty)$. Since $\delta<s$, we only need to prove $c(x)>0$, where
\begin{align*}
	c(x)&=bx^2+(-2b^2s-3bs+2bn-2b^2\delta-3b\delta-3b-3)x\\
	&\ \ \ +(b^2s+bs-bn+b^2\delta+b\delta-2b^2-b+1)s\\
	&\ \ \ +(-bn+b^2\delta+b\delta-2b^2-b+1)\delta-n+2bn-4b-2.
\end{align*}	
By direct calculation, the symmetry axis of $c(x)$ is
\begin{align*}
	\hat{x}&=-\frac{(-2b^2s-3bs+2bn-2b^2\delta-3b\delta-3b-3)}{2b}\\
	&=\frac{3}{2b}+\frac{3}{2}+(\frac{3}{2}+b)\delta+\frac{3}{2}s+bs-n\\
	&<3+(\frac{3}{2}+b)\delta+(\frac{3}{2}+b)s-n\\
	&\le3+(\frac{3}{2}+b)\delta+(\frac{3}{2}+b)\frac{n-2}{b+1}-n\\
	&=1+(\frac{3}{2}+b)\delta+\frac{1}{2(b+1)}n-\frac{1}{b+1}.
\end{align*}
Since $\delta\ge3$ and $n\ge2b\delta^2$, we have $n\ge2b\delta^2>2+\frac{2}{3}(3+2b)\delta>2+\frac{b+1}{2b+1}(3+2b)\delta$. Note that
$n>2+\frac{b+1}{2b+1}(3+2b)\delta\iff 1+(\frac{3}{2}+b)\delta+\frac{1}{2(b+1)}n-\frac{1}{b+1}<n-1$, we get $-\frac{(-2b^2s-3bs+2bn-2b^2\delta-3b\delta-3b-3)}{2b}<n-1$, which implies $c(x)$ is increasing with respect to $x\in[n-1,+\infty)$. Hence,
\begin{align*}
	c(x)\ge c(n-1)&=(b^2+b)s^2+(1+2b+b\delta+b^2\delta-4bn-2b^2n)s+3bn^2\\
	& \ \ \ +b(b+1)\delta^2+(1+2b-4bn-2b^2n)\delta-4n-5bn+1.
\end{align*}
Let
\begin{align*}
	h(s)\triangleq c(n-1)&=(b^2+b)s^2+(1+2b+b\delta+b^2\delta-4bn-2b^2n)s+3bn^2\\
	& \ \ \ +b(b+1)\delta^2+(1+2b-4bn-2b^2n)\delta-4n-5bn+1.
\end{align*}
Recall that $\delta<s\le \frac{n-2}{b+1}$ and $n\ge2b\delta^2$, we obtain
\begin{align*}
	\frac{dh}{ds}&=2b(b+1)s+1+2b+b\delta+b^2\delta-4bn-2b^2n\\
	&\le2b(n-2)+1+2b+b\delta+b^2\delta-4bn-2b^2n\\
	&=-2bn+1-2b+b\delta+b^2\delta-2b^2n<0.
\end{align*}
Thus, $h(s)$ is decreasing with respect to $s\in[\delta+1,\frac{n-2}{b+1}]$. By direct calculation,
\begin{align*}
	h(s)\ge h(\frac{n-2}{b+1})&=\frac{1}{b+1}[b^2n^2-(3+3b+b^2+(3b+5b^2+2b^3)\delta)n\\
	& \ \ \ +b^3\delta^2+2\delta^2b^2+b\delta^2+b^2\delta+(b+1)\delta+b-1]\\
	& \ \ \ >\frac{1}{b+1}[b^2n^2-(3+3b+b^2+(3b+5b^2+2b^3)\delta)n]\\
	& \ \ \ =\frac{b^2}{b+1}[n^2-((\dfrac{3}{b}+5+2b)\delta+1+\frac{3(b+1)}{b^2})n].
\end{align*}
Note that $n\ge\max\{ 2b\delta^2,(\dfrac{3}{b}+5+2b)\delta+1+\frac{3(b+1)}{b^2}\}\ge(\dfrac{3}{b}+5+2b)\delta+1+\frac{3(b+1)}{b^2}$. Hence $c(x)\ge c(n-1)=h(s)>0$, which implies $f_s(x)<f_{\delta}(x)$ for $x\in[n-1,+\infty)$. Recall that $\min\{\mu(G_s),\mu(G_{\delta})\}>n-1$, we have $\mu(G_s)>\mu(G_{\delta})$. Furthermore, combining with (\ref{eq2}) and (\ref{e:2}), we get
$$\mu(G)\ge\mu(G_1)\ge\mu(G_s)>\mu(G_{\delta})=\mu(K_{\delta} \vee (K_{n-(b+1)\delta-1}+(b\delta+1)K_1),$$
a contradiction.

{\bf Case 3.} $1\le s<\delta$.

Since $G$ is a spanning subgraph of $G_1=K_s \vee (K_{n_1}+K_{n_2}+\cdots+ K_{n_t})$, where $n_1\ge n_2\ge \cdots \ge n_t$ is odd integer, $t=bs+2$ and $n_1+n_2+\cdots+n_t=n-s$. It is easy to see that $\delta(G_1)\ge \delta(G)=\delta$, we have $n_t-1+s\ge \delta$. Thus, $n_1\ge n_2\ge \cdots \ge n_t\ge \delta-s+1$. Then we will proof that $n_1\ge 2(\delta-s+1)$. If $n_1< 2(\delta-s+1)$, then $n_1\le 2\delta-2s+1$. Since $n_1\ge n_2\ge \cdots \ge n_t$ and $1\le s<\delta$ and $\delta\ge3$, we obtain
\begin{align*}
	n&=s+n_1+n_2+\cdots+n_t\\
	&\le s+(bs+2)(2\delta-2s+1)\\
	&=-2bs^2+(-3+b+2b\delta)s+4\delta+2\\
	&\le -2b(\frac{\delta}{2}+\frac{1}{4}-\frac{3}{4b})^2+(-3+b+2b\delta)(\frac{\delta}{2}+\frac{1}{4}-\frac{3}{4b})+4\delta+2\\
	&=\frac{b\delta^2}{2}+\frac{b+5}{2}\delta+\frac{b}{8}+\frac{9}{8b}+\frac{5}{4}\\
	&<\frac{b\delta^2}{2}+\frac{b+5}{2}\delta+\frac{b}{8}+3.
\end{align*}
Let $$l(b)=2b\delta^2-(\frac{b\delta^2}{2}+\frac{b+5}{2}\delta+\frac{b}{8}+3)=\frac{3}{2}b\delta^2-(\frac{b+5}{2}\delta+\frac{b}{8}+3).$$
Note that $l'(b)=\frac{3}{2}\delta^2-\frac{1}{2}\delta-\frac{1}{8}>0$, we have $l(b)\ge l(1)=\frac{3}{2}\delta^2-3\delta-\frac{25}{8}>0$. Thus $n<2b\delta^2$.
This is a contradiction with $n\ge2b\delta^2$. Hence, $n_1\ge 2(\delta-s+1)$. Let $G_s=K_s\vee (K_{n-s-(\delta+1-s)(bs+1)}+(bs+1)K_{\delta+1-s})$. By Lemma \ref{le:4},
\begin{equation}\label{e:3}
	\mu(G_1)\ge \mu(G_s),
\end{equation}
where equality holds if and only if $G_1\cong G_s$. In what follows, we will discuss three subcases by classifying the value of $s$.

{\bf Case 3.1.} $s=1$.

In this case, $G_s=K_1\vee (K_{n-1-\delta(b+1)}+(b+1)K_{\delta})$, and the equitable quotient matrix of its distance matrix is
\begin{equation*}
	M_1=\begin{bmatrix}
		0 & n-(b+1)\delta-1 & (b+1)\delta \\
		1 & n-(b+1)\delta-2 & 2(b+1)\delta \\
		1 & 2(n-(b+1)\delta-1) & 2b\delta+\delta-1
	\end{bmatrix}.
\end{equation*}
By a simple calculation, the characteristic polynomial of $M_1$ is
\begin{align*}
	f_1(x)&=x^3+(3-b\delta-n)x^2+(3+3\delta+b\delta+3\delta^2+5b\delta^2+2b^2\delta^2-2n-3\delta n-2b\delta n)x\\
	&\ \ \ \ +(b^2+3b+2)\delta^2
	+(-bn-2n+b+2)\delta-n+1.
\end{align*}
Recall that in Case 2, by replacing $s$ with $\delta$, we can get the equitable quotient matrix $M_{\delta}$ of $G_{\delta}=K_{\delta} \vee (K_{n-(b+1)\delta-1}+(b\delta+1)K_1)$. Thus, the characteristic polynomial of $M_{\delta}$ is
\begin{align*} f_\delta(x)&=x^3+(-b\delta-n+3)x^2+(2b^2\delta^2+3b\delta^2-2bn\delta+3b\delta+3\delta-5n+6)x-(b^2+b)\delta^3\\
	&\ \ \ \ +(bn+2b^2+b-1)\delta^2+(n-2bn+4b+2)\delta+4-4n.
\end{align*}
Since $\delta\ge3$ and $n\ge 2b\delta^2$, for $x\in[n-1,+\infty)$, we have
\begin{align*}
	f_{\delta}(x)-f_1(x)&=[3n(\delta-1)-2b\delta^2-3\delta^2+2b\delta+3]x\\
	&\ \ \  -(b^2+b)\delta^3+(bn+b^2-2b-3)\delta^2+(-bn+3n+3b)\delta-3n+3\\
	&\ \ \ \ge[3n(\delta-1)-2b\delta^2-3\delta^2+2b\delta+3](n-1)\\
	&\ \ \ \ \  -(b^2+b)\delta^3+(bn+b^2-2b-3)\delta^2+(-bn+3n+3b)\delta-3n+3\\
	&\ \ \ =(\delta-1)[3n^2-(b\delta+3\delta+3)n-b^2\delta^2-b\delta^2-b\delta]\\
	&\ \ \  \triangleq (\delta-1)m(n).
\end{align*}	
Observe that $G_{\delta}$ and $G_s$ are spanning subgraphs of $K_n$, by lemma \ref{le:8}, $\mu(G_{\delta})>\mu(K_{n})=n-1$ and $\mu(G_s)>\mu(K_{n})=n-1$.
Then we will prove that $f_{\delta}(x)-f_1(x)>0$ for $x\in[n-1,+\infty)$. Since $\delta\ge3$, we only need to prove $m(n)>0$.

The symmetry axis of $m(n)$ is
\begin{align*}
	\hat{n}&=\frac{b\delta+3\delta+3}{6}<2b\delta^2.
\end{align*}
Thus, $m(n)$ is increasing with respect to $n\in[2b\delta^2,+\infty)$. By a simple calculation, we have
$$m(n)\ge m(2b\delta^2)=b\delta[12b\delta^3-(2b+6)\delta^2-(7+b)\delta-1]\triangleq b\delta h(b).$$
It is easy to see that $h(b)$ is increasing with respect to $b\in[1,+\infty)$. Thus $h(b)\ge h(1)=\delta(12\delta^2-8\delta-8)-1>0$ and $m(n)\ge m(2b\delta^2)=b\delta h(b)>0$, which implies $f_{\delta}(x)>f_1(x)$ for $x\in[n-1,+\infty)$. Note that $\min\{\mu(G_{\delta}),\mu(G_s)\}>n-1$, we have $\mu(G_s)>\mu(G_{\delta})$. Furthermore, by (\ref{eq2}) and (\ref{e:3}),
$$\mu(G)\ge\mu(G_1)\ge\mu(G_s)>\mu(G_{\delta})=\mu(K_{\delta} \vee (K_{n-(b+1)\delta-1}+(b\delta+1)K_1),$$
a contradiction.

{\bf Case 3.2.} $2\le s\le \delta-1$.

Note that $G_s=K_s\vee (K_{n-s-(\delta+1-s)(bs+1)}+(bs+1)K_{\delta+1-s})$. The distance matrix $D(G_s)$ of $G_s$ is
\begin{equation*}
	\bordermatrix{%
	&K_s &K_{n-s-(\delta+1-s)(bs+1)} &K_{\delta+1-s} &\cdots &K_{\delta+1-s}\cr
K_s &J-I &J &J &\cdots &J\cr
K_{n-s-(\delta+1-s)(bs+1)} &J &J-I &2J &\cdots &2J\cr
K_{\delta+1-s} &J &2J &J-I &\cdots &2J\cr
\vdots &\vdots &\vdots &\vdots &\vdots &\vdots\cr
K_{\delta+1-s} &J &2J &2J &\cdots &J-I\cr
	},
\end{equation*}
where $J$ denotes the all-one matrix and $I$ denotes the identity square matrix. Then we use $P_s$ to denote the equitable quotient matrix of the distance matrix $D(G_s)$ for the partition $V(K_s)\cup V(K_{n-s-(\delta+1-s)(bs+1)})\cup V((bs+1)K_{\delta+1-s})$. Thus
\begin{equation*}
	P_s=\begin{bmatrix}
		s-1 & n-s-(\delta+1-s)(bs+1) & (\delta+1-s)(bs+1) \\
		s & n-s-(\delta+1-s)(bs+1)-1 & 2(\delta+1-s)(bs+1) \\
		s & 2(n-s-(\delta+1-s)(bs+1)) & 2bs(\delta+1-s)+(\delta-s)
	\end{bmatrix}.
\end{equation*}
and the characteristic polynomial of $P_s$ is
\begin{align*}
	f_s(x)&=x^3+(3-n-bs-b\delta s+bs^2)x^2\\
	&\ \ \ \ +[2 b^2 s^4+(2b-4b^2-4b^2\delta)s^3+(-5 b + 2 b^2 - 7 b\delta + 4 b^2\delta + 2 b^2 \delta^2 + 2 b n) s^2\\
	&\ \ \ \ +(-3 + 3 b - 3\delta + 8 b\delta + 5 b\delta^2 + 3 n - 2 b n - 2 b\delta n)s-3\delta n-5n+3\delta^2+6\delta+6]x \\
	&\ \ \ \ -b^2 s^5+(-b+4b^2+2b^2\delta)s^4+(5b-5b^2+3b\delta-6b^2\delta-b^2 \delta^2 -b n)s^3\\
	&\ \ \ \  +(1-8b+2b^2+\delta-11b\delta +4 b^2\delta-2b \delta^2+2b^2\delta^2-n+ 3bn+b\delta n)s^2\\
	&\ \ \ \ +(-4 + 4 b - 5\delta + 9 b\delta - \delta^2 + 5 b \delta^2 + 4 n - 2 b n + \delta n -
	2 b \delta n) s\\
	&\ \ \ \ -3\delta n-4n+3\delta^2+6\delta+4.
\end{align*}
We use $y_1(M_s)$ to denote the largest real root of the equation $f_s(x)=0$. By Lemma \ref{le:10}, we have that $\mu(G_s)=y_1(M_s)$. Recall that in  Case 3.1, we can get the equitable quotient matrix $M_{\delta}$ of $G_{\delta}=K_{\delta} \vee (K_{n-(b+1)\delta-1}+(b\delta+1)K_1)$. Thus, the characteristic polynomial of $M_{\delta}$ is
	\begin{align*} f_\delta(x)&=x^3+(-b\delta-n+3)x^2+(2b^2\delta^2+3b\delta^2-2bn\delta+3b\delta+3\delta-5n+6)x-(b^2+b)\delta^3\\
		&\ \ \ \ +(bn+2b^2+b-1)\delta^2+(n-2bn+4b+2)\delta+4-4n,
	\end{align*}
and
	\begin{align*}
		f_{\delta}(x)-f_{s}(x)&=(\delta-s)[(-b+bs)x^2+(-3+3b-3\delta+3b\delta+2b^2\delta+ 3n-2bn\\
		&\ \ \ +(-5b+2b^2- 5b\delta +2bn)s+(2b-4b^2-2b^2\delta )s^2+2b^2s^3)x\\
		&\ \ \ -b^2s^4+(-b+4b^2+b^2\delta)s^3+(5 b-5b^2+2b\delta-2b^2\delta-bn)s^2\\
		&\ \ \ +(1-8b+2b^2+\delta-6b\delta-b^2 \delta-n+3bn)s\\
		&\ \ \ +(b\delta-2b+4)n-(b^2+b)\delta^2+(2b^2+b-4)\delta+4b-4]\\
		&\ \ \ \triangleq(\delta-s)H(x).
	\end{align*}	
Since $s\le \delta-1$, then we will prove that $H(x)>0$. In what follows, we will show $H(x)>0$ in two steps.

{\bf Step 1.} $H(n-1)>0$.
\small{\begin{align*}
	H(n-1)&=3(1- b+bs)n^2\\
	&\ \ \ +(2b^2s^3+(b-4b^2-2b^2\delta)s^2+(-1-6b+2b^2-5b\delta)s+(2b^2+4b-3)\delta+5b-2)n\\
	&\ \ \ -b^2s^4+(-b+2b^2+b^2\delta) s^3+(3 b-b^2+2b\delta)s^2+(1-2b+\delta-b\delta-b^2\delta)s\\
	&\ \ \ -(b^2+b)\delta^2-(2b+1)\delta-1.
\end{align*}}
Let $q(n)\triangleq H(n-1)$, then the symmetry axis of $q(n)$ is
\begin{align*}
	\hat{n}&=-\frac{2b^2s^3+(b-4b^2-2b^2\delta)s^2+(-1-6b+2b^2-5b\delta)s+(2b^2+4b-3)\delta+5b-2}{6(1- b+bs)}.
\end{align*}
Note that
\begin{align*}
	&\ \ \ -(2b^2s^3+(b-4b^2-2b^2\delta)s^2+(-1-6b+2b^2-5b\delta)s+(2b^2+4b-3)\delta+5b-2)\\
	&\ \ \
	=2+3\delta-4b\delta+s+5b\delta s+b(s-1)(5-s)-2b^2s(s-1)^2+2b^2\delta(s^2-1)\\
	&\ \ \
	 =2+3\delta-4b\delta+s+5b\delta s+b(s-1)[2bs(\delta-s)+2b\delta+2bs-s+5]>0.
\end{align*}	
Since $2\le s\le \delta-1$, we have
\begin{align*}
		\hat{n}&=-\frac{2b^2s^3+(b-4b^2-2b^2\delta)s^2+(-1-6b+2b^2-5b\delta)s+(2b^2+4b-3)\delta+5b-2}{6(1- b+bs)}\\
	&\ \ \
	<\frac{2+3\delta-4b\delta+s+5b\delta s+b(s-1)[2bs(\delta-s)+2b\delta+2bs-s+5]}{6b(s-1)}\\
	&\ \ \
	=\frac{2+3\delta-4b\delta+s+5b\delta s}{6b(s-1)}+\frac{5-s+2bs(\delta-s)+2b\delta+2bs}{6}\\
	&\ \ \
	<\frac{2+5b\delta s}{6b}+\frac{3+2bs\delta+2b\delta+2bs}{6}\\
	&\ \ \
	< \frac{2+5b\delta^2}{6b}+\frac{3+2b\delta^2+2b\delta+2b\delta}{6}\\
	&\ \ \
	<1+\frac{2}{3}b\delta+(\frac{5}{6}+\frac{b}{3})\delta^2.
\end{align*}	
Note that $\delta\ge3$, it can be checked that $1+\frac{2}{3}b\delta+(\frac{5}{6}+\frac{b}{3})\delta^2<2b\delta^2$, which implies that $q(n)$ is increasing with respect to $n\in[2b\delta^2,+\infty)$. By a simple calculation,
\begin{align*}
	q(n)&\ge q(2b\delta^2)\\
	&=12b^2(1+b(s-1))\delta^4+(-6 b+ 8 b^2 + 4 b^3 - 10 b^2 s - 4 b^3 s^2)\delta^3\\
	&\ \ \ \ +(-5 b + 9 b^2 + (-2 b - 12 b^2 + 4 b^3) s + (2 b^2 - 8 b^3) s^2 +
	4 b^3 s^3)\delta^2\\
	&\ \ \ \ +(-1 - 2 b + (1 - b - b^2) s + 2 b s^2 + b^2 s^3)\delta \\
	&\ \ \ \ -1 + (1 - 2 b) s + (3 b - b^2) s^2 + (-b + 2 b^2) s^3 - b^2 s^4.
\end{align*}
Next, we will prove that $q(2b\delta^2)>0$ progressively scaling. Since $\delta\ge s+1$ and $s\ge2$, we have
\begin{align*}
	 	&\ \ \
	 	12b^2(1+b(s-1))\delta^4+(-6 b+ 8 b^2 + 4 b^3 - 10 b^2 s - 4 b^3 s^2)\delta^3\\
	& =\delta^3[12b^2(1+b(s-1))\delta+(-6 b+ 8 b^2 + 4 b^3 - 10 b^2 s - 4 b^3 s^2)]\\
	& \ge\delta^3[12b^2(1+b(s-1))(s+1)+(-6 b+ 8 b^2 + 4 b^3 - 10 b^2 s - 4 b^3 s^2)] \\
	&=\delta^3(-6 b + 20 b^2 - 8 b^3 + 2 b^2 s + 8 b^3 s^2)>0.
\end{align*}
Then
 \begin{align*}
 	&\ \ \
 	\delta^2[(-6 b + 20 b^2 - 8 b^3 + 2 b^2 s + 8 b^3 s^2)\delta +(-5 b + 9 b^2 + (-2 b - 12 b^2 + 4 b^3) s \\
 	&\ \ \ \ \ + (2 b^2 - 8 b^3) s^2 +
 	4 b^3 s^3)]\\
 	&\ge \delta^2[(-6 b + 20 b^2 - 8 b^3 + 2 b^2 s + 8 b^3 s^2)(s+1)\\
 	&\ \ \ \ +(-5 b + 9 b^2 + (-2 b - 12 b^2 + 4 b^3) s + (2 b^2 - 8 b^3) s^2 +
 	4 b^3 s^3)]\\
 	&=\delta^2(-11 b + 29 b^2 - 8 b^3 + (-8 b + 10 b^2 - 4 b^3) s + 4 b^2 s^2 +
 	12 b^3 s^3).
 \end{align*}

Therefore
  \begin{align*}
  	&\ \ \
  \delta[(-11 b + 29 b^2 - 8 b^3 + (-8 b + 10 b^2 - 4 b^3) s + 4 b^2 s^2 +
  	12 b^3 s^3)\delta\\
  	&\ \ \ \ +(-1 - 2 b + (1 - b - b^2) s + 2 b s^2 + b^2 s^3)]\\
  	& \ge \delta[(-11 b + 29 b^2 - 8 b^3 + (-8 b + 10 b^2 - 4 b^3) s + 4 b^2 s^2 +
  	12 b^3 s^3)(s+1)\\
  	&\ \ \ \ +(-1 - 2 b + (1 - b - b^2) s + 2 b s^2 + b^2 s^3)]\\
  	& =\delta[-1 - 13 b + 29 b^2 -
  	8 b^3 + (1 - 20 b + 38 b^2 - 12 b^3) s\\
  	&\ \ \ + (-6 b + 14 b^2 -
  	4 b^3) s^2 + (5 b^2 + 12 b^3) s^3 + 12 b^3 s^4].
  \end{align*}

Finally, we get
    \small{\begin{align*}
  	&\delta[-1 - 13 b + 29 b^2 -
  	8 b^3 + (1 - 20 b + 38 b^2 - 12 b^3) s + (-6 b + 14 b^2 -
  	4 b^3) s^2 + (5 b^2 + 12 b^3) s^3 + 12 b^3 s^4]\\
  	&+(-1 + (1 - 2 b) s + (3 b - b^2) s^2 + (-b + 2 b^2) s^3 - b^2 s^4)\\
  	& > -1 - 13 b + 29 b^2 -
  	8 b^3 + (1 - 20 b + 38 b^2 - 12 b^3) s + (-6 b + 14 b^2 -
  	4 b^3) s^2 + (5 b^2 + 12 b^3) s^3 + 12 b^3 s^4\\
  	&\ \ \ \  +(-1 + (1 - 2 b) s + (3 b - b^2) s^2 + (-b + 2 b^2) s^3 - b^2 s^4)\\
  	&=-2 - 13 b + 29 b^2 -
  	8 b^3 + (2 - 22 b + 38 b^2 - 12 b^3) s + (-3 b + 13 b^2 -
  	4 b^3) s^2\\
  	&\ \ \ + (-b + 7 b^2 + 12 b^3) s^3 + (-b^2 + 12 b^3) s^4\\
  	& >11 b^3s^4+12 b^3s^3-
  	4 b^3s^2- 12 b^3 s-8 b^3 >0.
  \end{align*}}
According to the above calculation process, we obtain $q(n)\ge q(2b\delta^2)>0$, which implies that $H(n-1)>0$.

{\bf Step 2.} $H'(x)>0$ for $x\in[n-1,+\infty)$.

Recall that
  \begin{align*}
  	H(x)&=(-b+bs)x^2+(-3+3b-3\delta+3b\delta+2b^2\delta+ 3n-2bn\\
  	&\ \ \ +(-5b+2b^2- 5b\delta +2bn)s+(2b-4b^2-2b^2\delta )s^2+2b^2s^3)x\\
  	&\ \ \ -b^2s^4+(-b+4b^2+b^2\delta)s^3+(5 b-5b^2+2b\delta-2b^2\delta-bn)s^2\\
  	&\ \ \ +(1-8b+2b^2+\delta-6b\delta-b^2 \delta-n+3bn)s\\
  	&\ \ \ +(b\delta-2b+4)n-(b^2+b)\delta^2+(2b^2+b-4)\delta+4b-4.
  \end{align*}	
Then
   \begin{align*}
  	H'(x)&=2(-b+bs)x+(-3+3b-3\delta+3b\delta+2b^2\delta+ 3n-2bn\\
  	&\ \ \ +(-5b+2b^2- 5b\delta +2bn)s+(2b-4b^2-2b^2\delta )s^2+2b^2s^3)\\
  	& \ge 2(-b+bs)(n-1)+(-3+3b-3\delta+3b\delta+2b^2\delta+ 3n-2bn\\
  	&\ \ \ +(-5b+2b^2- 5b\delta +2bn)s+(2b-4b^2-2b^2\delta )s^2+2b^2s^3)\\
  	&=
  	2 b^2 s^3+ (2 b -
  	4 b^2 - 2 b^2 \delta) s^2 + (-7 b + 2 b^2 - 5 b \delta + 4 b n) s\\
  	&\ \ \ + 2 b^2 \delta + (5 +3 \delta-4n) b  + 3( n - \delta-1)\\
  	& \triangleq g(s).
  \end{align*}	
Next we prove that $g(s)>0$ for $2\le s\le \delta-1$. By direct calculation, we deduce that
  $$g'(s)=6 b^2 s^2+b (4 - 4 b (2 + \delta)) s+b(-7 + 2 b - 5 \delta + 4 n),$$
and the symmetry axis of $g'(s)$ is $\hat{s}=\frac{\delta}{3}+\frac{2}{3}-\frac{1}{3b}$.

Since $n\ge 2b\delta^2$ and $\delta\ge 3$,
$$g'(s)\ge g'(\frac{\delta}{3}+\frac{2}{3}-\frac{1}{3b})=4 b n-\frac{1}{3} (2+13 b+11 b \delta + 2 b^2 (1 + 4 \delta + \delta^2)).$$
Note that
 \begin{align*}
	g'(s)&\ge4 b n-\frac{1}{3} (2+13 b+11 b \delta + 2 b^2 (1 + 4 \delta + \delta^2))\\
	& \ge\frac{1}{3}[24b^2\delta^2-(2+13 b+11 b \delta + 2 b^2 (1 + 4 \delta + \delta^2))]\\
	& =\frac{1}{3}[22b^2\delta^2-2-13 b-11 b \delta -2 b^2 -8b^2 \delta]\\
	& \triangleq \frac{1}{3} v(b).
\end{align*}	
For $b\in [1,+\infty)$, we have
$$v'(b)=44b\delta^2-13-11\delta-4b-16b\delta>0,$$
Thus, $v'(b)>0$ and $v(b)\ge v(1)=22\delta^2-17-19\delta>0$.

Therefore, $g'(s)>0$ and $g(s)$ is increasing with respect to $s\in[2,\delta-1]$. Hence
$$g(s)\ge g(2)=(3 + 4 b) n+4 (2 b - 4 b^2 - 2 b^2\delta)+(2b^2-7b-3)\delta+20b^2-9b-3.$$
Since $\delta\ge3$, it can be checked that
 \begin{align*}
	g(2)&=(3 + 4 b) n+4 (2 b - 4 b^2 - 2 b^2\delta)+(2b^2-7b-3)\delta+20b^2-9b-3\\
	& >(3 + 4 b) 2b\delta^2 - 4(4b^2+2b^2\delta)-(7b+3)\delta\\
	& >6b\delta^2+8b^2\delta^2-24b^2\delta-10b\delta\\
	& >18b\delta+24b^2\delta-24b^2\delta-10b\delta\\
	& =8b\delta>0.
\end{align*}	
Thus, we have that $H'(x)>0$ for $x\in[n-1,+\infty).$

Combining with Step 1 and Step 2, we get $H(x)>0$ for $x\in[n-1,+\infty)$, which implies $f_{\delta}(x)>f_s(x)$ for $x\in[n-1,+\infty)$. Observe that $\min\{\mu(G_{\delta}),\mu(G_s)\}>n-1$, we have $\mu(G_s)>\mu(G_{\delta})$.

Furthermore, by (\ref{eq2}) and (\ref{e:3}),
$$\mu(G)\ge\mu(G_1)\ge\mu(G_s)>\mu(G_{\delta})=\mu(K_{\delta} \vee (K_{n-(b+1)\delta-1}+(b\delta+1)K_1),$$
a contradiction.  This completes the proof.
\end{proof}

Note that a perfect matching is a special odd $[1,b]$-factor when $b=1$. Let $b=1$, then we can obtain a condition in terms of distance spectral radius about perfect matching with given minimum degree.

\noindent\begin{corollary}  \ Let $G$ be a connected graph of even order $n\ge\max\{2\delta^2,10\delta+7\}$ with minimum degree $\delta\ge3$. If $\mu(G)\le \mu(K_{\delta} \vee (K_{n-2\delta-1}+(\delta+1)K_1))$, then $G$ has a perfect matching unless $G\cong K_{\delta} \vee (K_{n-2\delta-1}+(\delta+1)K_1).$
\end{corollary}

\section*{Statement}

This article has undergone further revisions and improvements, all of which were contributed by Ligong Wang. In recognition of his contributions, Ligong Wang is now acknowledged as a new co-author.
The manuscript has been updated to its current version to incorporate these changes. We affirm that all authors have reviewed and approved this update.

%\section*{Declarations}

%The authors declare that they have no conflict of interests.

%\section*{Ethical Approval}

%Not applicable

%\section*{Funding}

%This work was supported by the National Natural Science Foundation of China (Nos. 12001434 and 12271439), the Natural Science Basic Research Program of Shaanxi Province (Nos. 2022JM-006 and 2023-JC-YB-070) and Chinese Universities Scientific Fund (No. 2452020021).

%\section*{Availability of data and materials}

%Not applicable

%\section*{Author Contributions Statement}

%Z.Z. Xu obtained the main results, W.G. Xi improved the manuscript. All authors reviewed the manuscript.


\begin{thebibliography}{99}
	
\bibitem{A}	A. Amahashi, On factors with all degrees odd, Graphs Combin., 1 (1985) 111-114.
	
%\bibitem{AK} A. Amahashi, M. Kano, On factors with given components, Discrete Math., 42 (1982) 1-6.
		
\bibitem{FLL} D.D. Fan, H.Q. Lin, H.L. Lu, Spectral radius and $[a, b]$-factors in graphs, Discrete Math., 345 (2022) 112892.
	
\bibitem{FYZ} L.H. Feng, G.H. Yu, X.D. Zhang, Spectral radius of graphs with given matching number, Linear Algebra Appl., 422 (2007) 133-1382.
	
\bibitem{G} C.D. Godsil, Algebraic Combinatorics, Chapman and Hall Mathematics Series, New York, 1993.

\bibitem{GLMM}	M.Y. Guo, H.L. Lu, X.X. Ma, X. Ma, Spectral radius and rainbow matchings of graphs, Linear Algebra Appl., 679 (2023) 30-37.
	
\bibitem{HLL} Y.F. Hao, S.C. Li, X.C. Li, Vertex cut, eigenvalues, $[a,b]$-factors and toughness of connected bipartite graphs, Discrete Math., 347 (2024) 114118.
	
\bibitem{HLZ} Y.F. Hao, S.C. Li, Q. Zhao, On the $A_{\alpha}$-spectral radius of graphs without large matchings, Bull. Malays. Math. Sci. Soc., 45 (2022) 3131-3156
	
\bibitem{Il}  A. Ili\'{c}, Distance spectral radius of trees with given matching number, Discrete Appl. Math., 158 (2010) 1799-1806.
	
\bibitem{KOSS} M. Kim, S. O, W. Sim, D. Shin, Matchings in graphs from the spectral radius, Linear  Multilinear Algebra, 71 (2023) 1794-1803.
	
\bibitem{LM} S.C. Li, S.J. Miao, Complete characterization of odd factors via the size, spectral radius or distance spectral radius of graphs, Bull. Korean Math. Soc., 59 (2022) 1045-1067.
	
\bibitem{LYL} C. Liu, Z.M. Yan, J.P. Li, The maximum $A_{\alpha}$-spectral radius of $t$-connected graphs with bounded matching number, Discrete Math., 346 (2023) 113447.
	
\bibitem{LLS} J. Lou, R.F. Liu, J.L. Shu, Toughness and distance spectral radius in graphs involving minimum degree, Discrete Appl. Math., 361 (2025) 34-47.	

\bibitem{LP} L. Lov\'{a}sz, M.D. Plummer, Matching theory, North-Holland-Amsterdam, New York, 1986.	
	
\bibitem{O} S. O, Spectral radius and matchings in graphs, Linear Algebra Appl., 614 (2021) 316-324.
	
\bibitem{O2} S. O, Eigenvalues and $[a, b]$-factors in regular graphs, J. Graph Theory, 100 (2022) 458-469.

\bibitem{SYZL} Y. Shen, L.H. You, M.J. Zhang, S.C. Li, On a conjecture for the signless Laplacian spectral radius of cacti with given matching number, Linear Multilinear Algebra, 65 (2017) 457-474.
	
\bibitem{YYSX} L.H. You, M. Yang, W. So, W.G. Xi, On the spectrum of an equitable quotient matrix and its application, Linear Algebra Appl., 577 (2019) 21-40.

\bibitem{Z} W.Q. Zhang, The maximum spectral radius of $t$-connected graphs with bounded matching number, Discrete Math., 345 (2022) 112775.

\bibitem{ZD} Y.K. Zhang, E.R. van Dam, Matching extension and distance spectral radius, Linear Algebra Appl., 674 (2023) 244-255.

\bibitem{ZL2} M.J. Zhang, S.C. Li, Extremal cacti of given matching number with respect to the distance spectral radius, Appl. Math. Comput., 291 (2016) 89-97.

\bibitem{ZL} Y.K. Zhang, H.Q. Lin, Perfect matching and distance spectral radius in graphs and bipartite graphs, Discrete Appl. Math., 304 (2021) 315-322.

\bibitem{ZL4} S.Z. Zhou, H.X. Liu, Two sufficient conditions for odd $[1,b]$-factors in graphs, Linear Algebra Appl., 661 (2023) 149-162.

\bibitem{ZL3} S.Z. Zhou, H.X. Liu, Characterizing an odd $[1,b]$-factor on the distance signless Laplacian spectral radius, RAIRO-Oper. Res., 57 (2023) 1343-1351.

\end{thebibliography}
\end{document}